\documentclass[12pt, reqno]{amsart}

\usepackage{amssymb}
\usepackage{thmtools}
\usepackage{subcaption}

\usepackage[shortlabels]{enumitem}
\setlist[enumerate]{leftmargin=*,font=\upshape,align=parleft,label=(\alph*)}
\setlist[itemize]{leftmargin=*,labelwidth=*}

\usepackage{microtype}

\usepackage[alphabetic, nobysame]{amsrefs}

\usepackage{hyperref}
\hypersetup{
	pdfstartview={XYZ null null 1.00}, 
	pdfpagemode=UseNone, 
	colorlinks,
	breaklinks, 
	linkcolor=blue,
	urlcolor=blue, 
	anchorcolor=blue,
	citecolor=blue
}

\usepackage[capitalise]{cleveref}

\usepackage{geometry}
\usepackage{xspace}
\usepackage{ bbold }

\usepackage{graphicx}
\usepackage{float}
\usepackage{xcolor,colortbl}

\makeatletter

\renewcommand\subsection{\@startsection{subsection}{2}%
	\z@{-1.5em}{.7em}%
	{\noindent\bfseries}}
\makeatother

\numberwithin{equation}{section}

\newtheorem{lemma}[equation]{Lemma}

\crefname{prop}{Proposition}{Propositions}
\newtheorem{prop}[equation]{Proposition}

\newtheorem{theorem}[equation]{Theorem}

\crefname{obs}{Observation}{Observations}
\newtheorem{obs}[equation]{Observation}

\crefname{cor}{Corollary}{Corollaries}
\newtheorem{cor}[equation]{Corollary}

\makeatletter
\def\@empty{}
\def\ifemptycredit#1{%
\def\tmp{#1}%
\ifx\tmp\@empty%
\else%
{~(#1)}%
\fi%
}
\makeatother

\newenvironment{namedthm*}[2][]{
\par\noindent \textbf{#2}\ifemptycredit{#1}\textbf{.}\itshape\xspace
}{}

\theoremstyle{definition}

\crefname{defn}{Definition}{Definitions}
\newtheorem{defn}[equation]{Definition}

\crefname{example}{Example}{Examples}
\newtheorem{example}[equation]{Example}

\crefname{problem}{Problem}{Problems}
\newtheorem{problem}[equation]{Problem}

\theoremstyle{remark}

\crefname{remark}{Remark}{Remarks}

\newtheorem*{notation*}{Notation}

\graphicspath{ {images/} }

\renewcommand{\phi}{\varphi}

\newcommand*{\defeq}{\mathrel{\vcenter{\baselineskip0.5ex \lineskiplimit0pt \hbox{\scriptsize.}\hbox{\scriptsize.}}}=}

\newcommand{\set}[1]{\left\{ #1 \right\}}

\newcommand{\floor}[1]{{\lfloor #1 \rfloor}}
\newcommand{\ceil}[1]{{\lceil #1 \rceil}}

\definecolor{bleu-fonce}{RGB}{90,94,176}
\definecolor{green-fonce}{RGB}{90,176,94}
\definecolor{actual-purple}{RGB}{140,20,160}

\title{Pairs of disjoint matchings and related classes of graphs}

\author{Huizheng Guo}
\author{Kieran Kaempen} 
\author{Zhengda Mo}
\author{Sam Qunell}
\author{Joe Rogge}
\author{Chao Song}
\author{Anush Tserunyan}
\author{Jenna Zomback}

\thanks{This work is part of the research project ``Pairs of disjoint matchings'' within \href{https://math.illinois.edu/research/igl}{Illinois Geometry Lab} in Spring 2019 -- Spring 2020. The first, second, third, fourth, fifth, and sixth authors participated as undergraduate scholars, the eighth author served as graduate student team leader, and the seventh author as faculty mentor. The seventh author was supported by the NSF Grant DMS-1855648.}

\date{\today}

\begin{document}

\maketitle

\begin{abstract}
    For a finite graph $G$, we study the maximum $2$-edge colorable subgraph problem and a related ratio $\frac{\mu(G)}{\nu(G)}$, where $\nu(G)$ is the matching number of $G$, and $\mu(G)$ is the size of the largest matching in any pair $(H,H')$ of disjoint matchings maximizing $|H| + |H'|$ (equivalently, forming a maximum $2$-edge colorable subgraph). Previously, it was shown that $\frac{4}{5} \le \frac{\mu(G)}{\nu(G)} \le 1$, and the class of graphs achieving $\frac{4}{5}$ was completely characterized. We show here that any rational number between $\frac{4}{5}$ and $1$ can be achieved by a connected graph. Furthermore, we prove that every graph with ratio less than $1$ must admit special subgraphs.
    

\end{abstract}

\section{Introduction}

We investigate the ratio $\frac{\mu(G)}{\nu(G)}$ of two parameters of a finite graph $G$, where $\nu(G)$ is the size of a maximum matching in $G$, and $\mu(G)$ is the size of the largest matching in a pair of disjoint matchings whose union is as large as possible. More formally,
\begin{align*}
\lambda(G) 
\defeq& 
\max\set{|H|+|H'| : \text{$H$ and $H'$ are disjoint matchings in $G$}},
\\
\mu(G)
\defeq&
\max\set{|H|:\text{$H$ and $H'$ are disjoint } \text{matchings in $G$ with }|H|+|H'|=\lambda(G)}.
\end{align*}

One motivation for studying the ratio $\frac{\mu(G)}{\nu(G)}$ is the problem
of covering as many edges of $G$ as possible by $k$ matchings, which is known as the Maximum $k$-Edge-Colorable Subgraph problem.
Let $\nu_k(G)$ denote the maximum number of edges of $G$ that can be covered by
$k$ matchings. When $k = 1$, it is clear that
one just takes a maximum matching, so $\nu_1(G) = \nu(G)$.
We consider the problem with $k = 2$, as $\nu_2(G) = \lambda(G)$.

\subsection{Computational Considerations}


Recall that a proper edge coloring of a graph $G$ is an assignment of a color to
each edge of $G$ such that no two adjacent edges have the same color.
In a proper edge coloring, the set of edges with a given color is by definition
a matching, so a proper edge coloring with $k$ colors corresponds exactly
to covering $G$ with $k$ matchings. We can therefore reframe $\nu_k$ to be
the maximum size in edges of a subgraph of $G$ that is $k$-edge-colorable.

Edge coloring is a computationally complex problem; in \cite{Holyer},
I. Holyer showed that the problem of determining whether a graph admits a $3$-edge coloring is $NP$-complete even in the case of cubic graphs. Note that a cubic graph is $3$-edge colorable if and only if
$\lambda$ is equal to the number of vertices, and so Holyer's result is equivalent to the $NP$-completeness of the problem of determining whether $\lambda$ is equal to the number of vertices.

Furthermore, U. Feige, E. Ofek, and U. Wieder showed in \cite{F-O-W:02} that
for all $k \ge 2$, the $k$-edge-colorable subgraph problem (computing $\nu_k$)
is $NP$-hard (note that there are polynomial time algorithms computing $\nu_1$,
the first being due to J. Edmonds in \cite{Edmonds:1965}).
Therefore, unless $P=NP$, we cannot compute $\lambda(G)$ in polynomial time.
In \cite{A-K-S-S-T:2020}, however, A. Agrawal et. al show that computing
$\lambda(G)$ is fixed-parameter tractable.

A greedy algorithm to compute $k$ matchings realizing $\nu_k(G)$
could first take a maximum matching of
the whole graph, then recursively take a maximum matching of the remaining edges until $k$ matchings have been computed. However, this may not yield a cover with maximum number of edges $\nu_k(G)$; see \cref{fig:spanner} for a minimal example. The parameter $\frac{\mu(G)}{\nu(G)}$ measures the failure of the optimality of this greedy algorithm for $k = 2$.

\begin{figure}[H]
    \centering
    \includegraphics[width=0.5\textwidth]{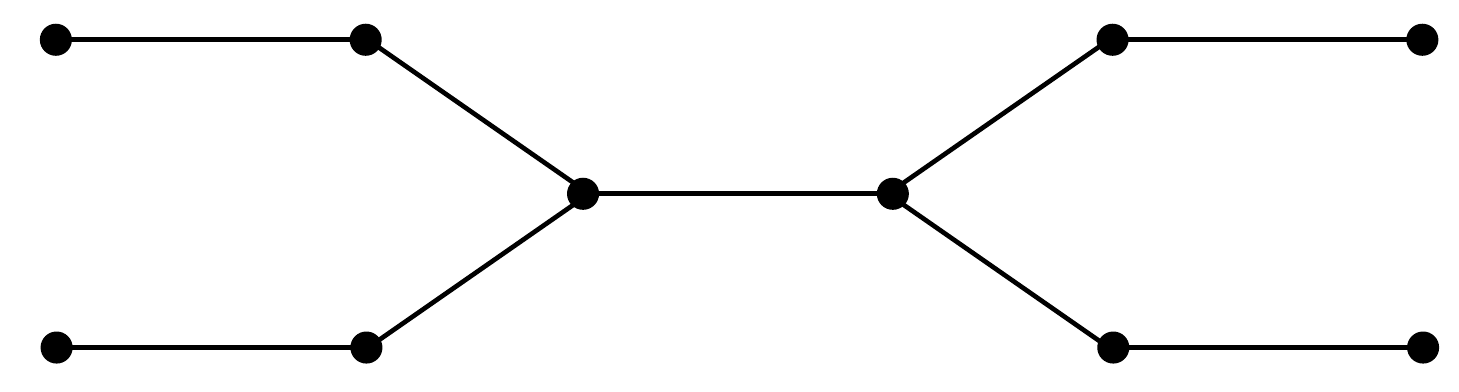}
    \caption{The spanner graph.}
    \label{fig:spanner}
\end{figure}

\subsection{Numerical Bounds}

In \cite{Mkrtchyan:2008qr}, V. Mkrtchyan, V. Musoyan, and A. Tserunyan proved that
\[ 
\frac{4}{5} \leq \frac{\mu(G)}{\nu(G)} \leq 1
\] 
for any graph $G$, and A. Tserunyan, in \cite{Tserunyan:2009qr}, characterized
all graphs that achieve the ratio $ \frac{4}{5} $---the characterization is
somewhat involved, see \cref{app:char} for a brief exposition. Figure
$\ref{fig:spanner}$
shows a picture of the \textbf{spanner}, the unique minimal graph that achieves
the ratio $ \frac{4}{5}$.

In \cite{A-M-P-V:2014} D. Aslanyan et. al. show that
in the special case of cubic graphs, 
\[ 
\frac{8}{9} \leq \frac{\mu(G)}{\nu(G)} \leq 1.
\] 

\noindent Furthermore, in \cite{M-P-V:10}, it is shown that in cubic graphs, the
following inequalities hold:
\begin{align*}
    \frac{4}{5} | V(G) | \le \lambda(G) \le \frac{|V(G)| + 2\nu_3(G)}{4}.
\end{align*}

\noindent In this vein, L. Karapetyan and V. Mkrtchyan prove in \cite{K-M:19} the following inequality: for a bipartite graph $G$, for all $0 \le i \le k$,
\begin{align*}
    \nu_k \ge \frac{\nu_{k-i}(G) + \nu_{k+i}(G)}{2}.
\end{align*}

\subsection{Our Results}

The present paper deals with graphs achieving $\frac{\mu(G)}{\nu(G)}>\frac{4}{5}$.
In \cref{sec:ratios}, we prove that for a connected graph $G$, the ratio $\frac{\mu(G)}{\nu(G)}$ achieves any rational number within our bounds. In fact:

\begin{theorem}\label{achieving_every_rational}
    For any positive integers $ m, n$ such that $ \frac{4}{5} \leq \frac{m}{n} < 1 $, there is a connected graph $ G $ with $ \mu(G) = m $ and $ \nu(G) = n $.
\end{theorem}

\begin{figure}[H]
            \centering
            \hspace*{\fill}
            \begin{subfigure}[b]{0.3\textwidth}
		     \includegraphics[scale=0.25]{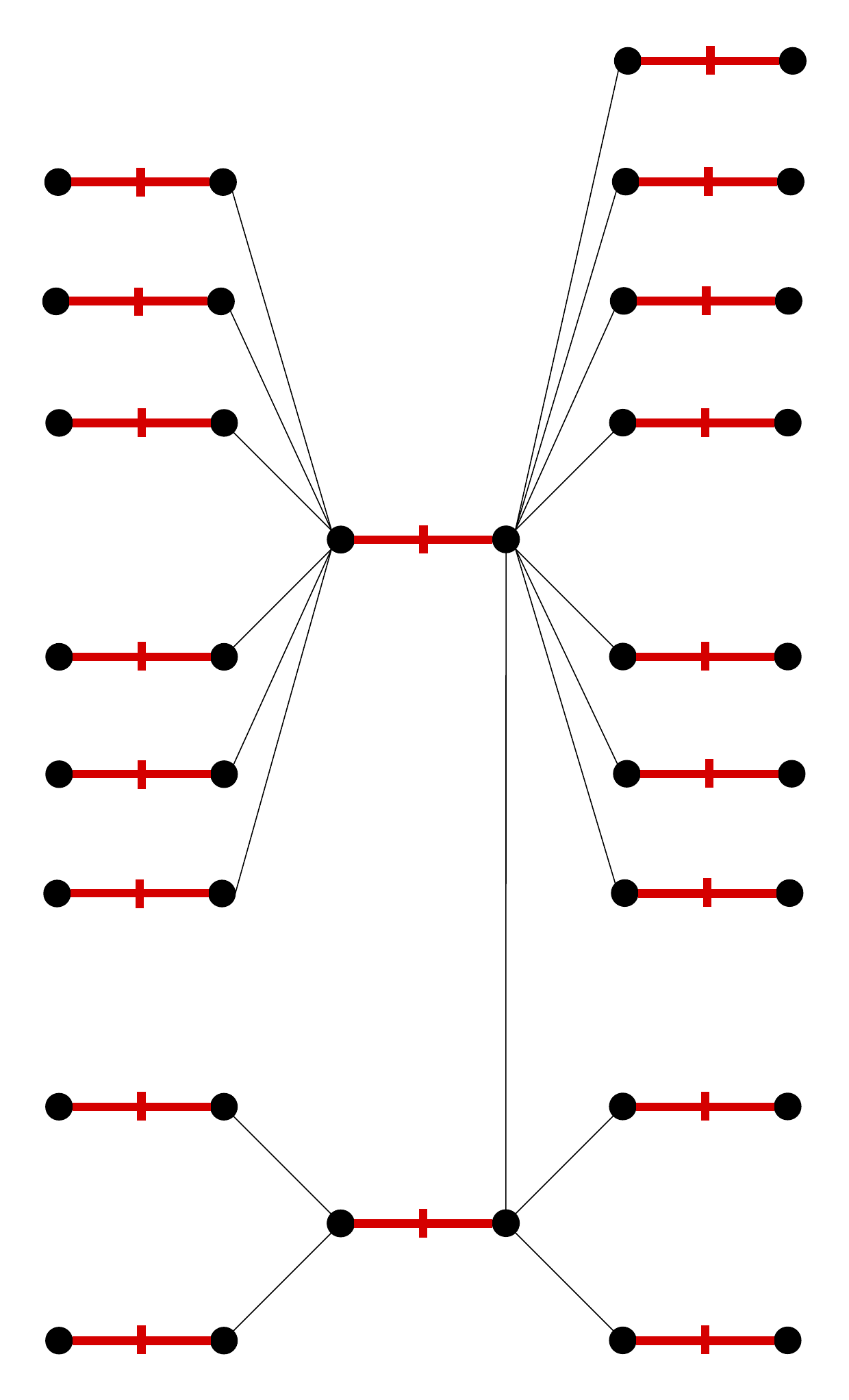}
		     \caption{A perfect matching.}
		     \label{fig:span_perf}
            \end{subfigure}
            \hfill
            \begin{subfigure}[b]{0.3\textwidth}
		     \includegraphics[scale=0.25]{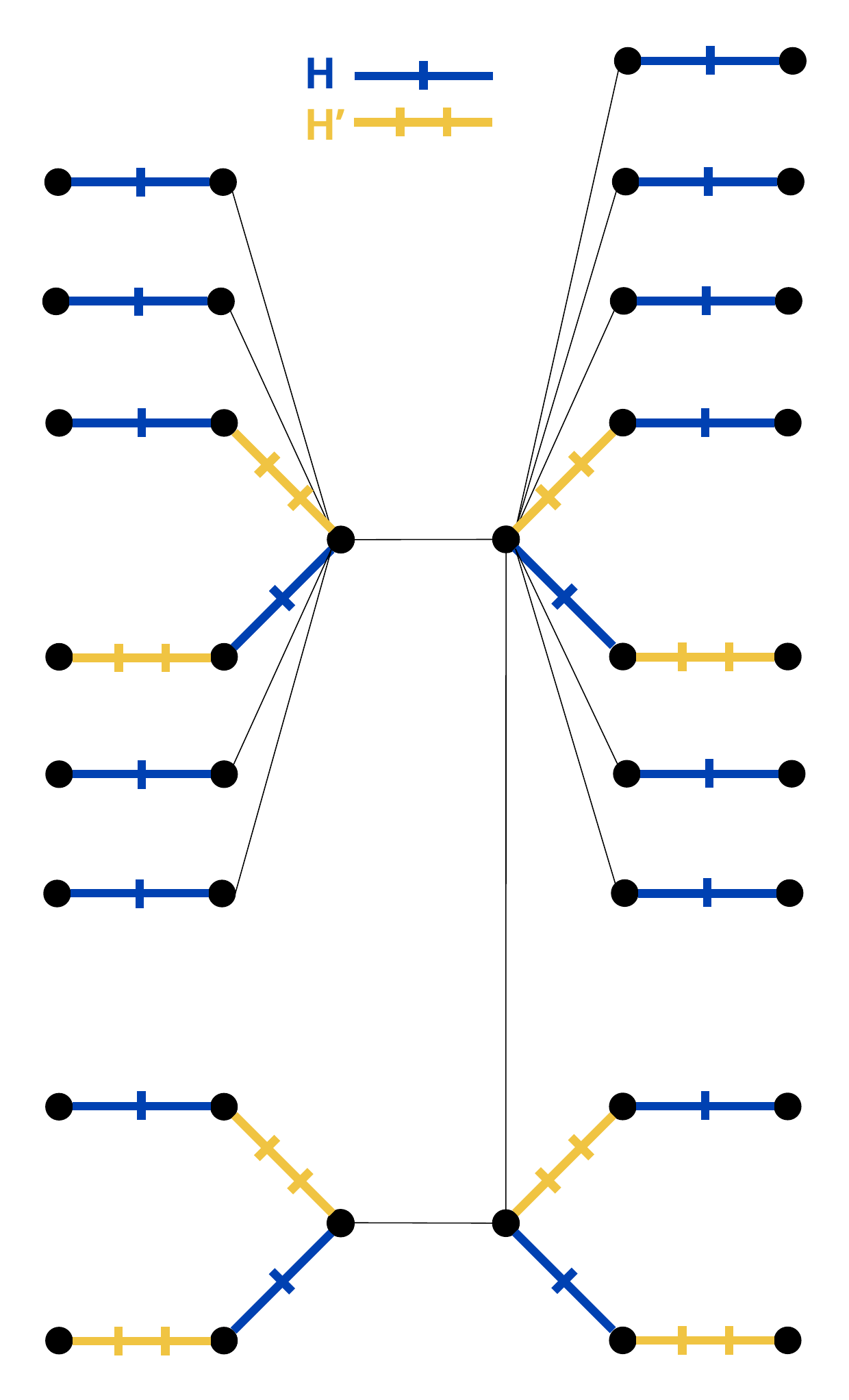}
		     \caption{A pair $(H, H')$.}
            \end{subfigure}
            \hspace*{\fill}
		     \caption{A graph $G$ with $\mu(G)=17$ and $\nu(G)=19$.}
\end{figure}

In \cref{sec:ratio_one}, we prove the following necessary structural condition for graphs with $\frac{\mu(G)}{\nu(G)}<1$.

\begin{theorem}\label{admitting_diamond_spanners}
If a finite connected graph $G$ has $\frac{\mu(G)}{\nu(G)}<1$, then it
contains a diamond spanner as a subgraph $($see \cref{fig:forbidden},
\cref{def:diamond_spanner}$)$.
\end{theorem}

\begin{figure}[H]
            \centering
		     \includegraphics[scale=.5]{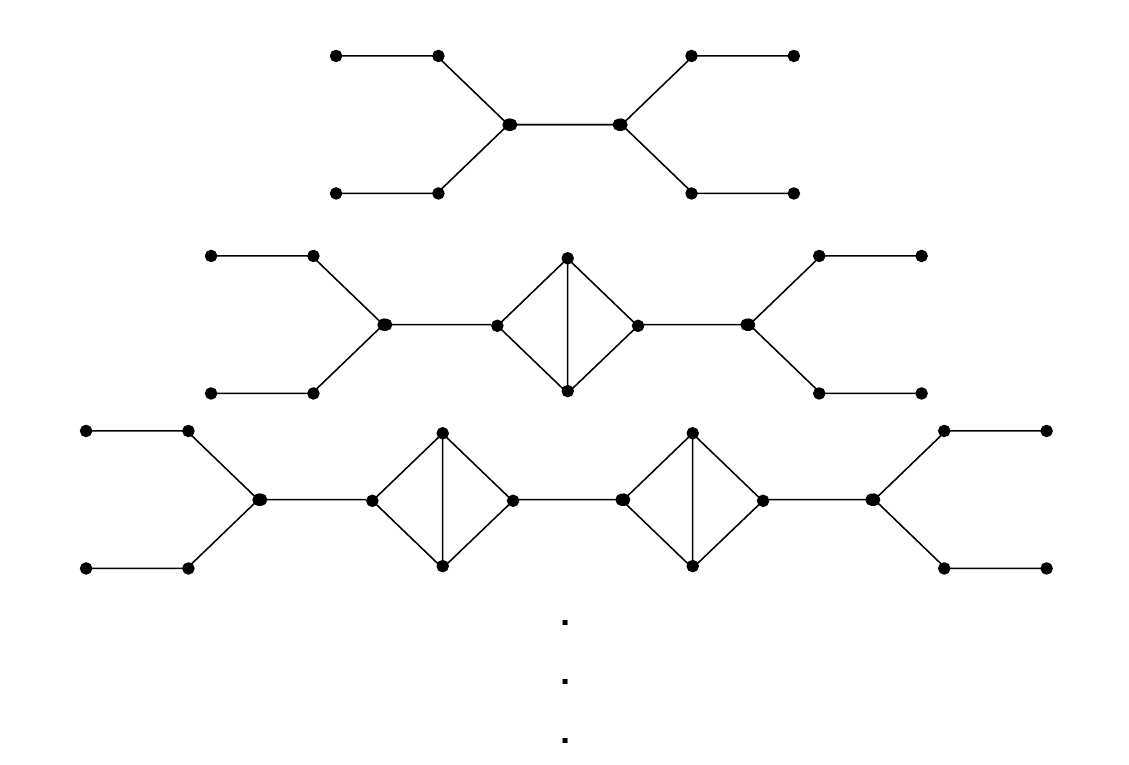}
		     \caption{Diamond spanners with 0, 1, and 2 diamonds.}
             \label{fig:forbidden}
\end{figure}

Unlike the lower bound $\frac{4}{5}$, we still do not know exactly for which graphs $G$ the ratio $\frac{\mu(G)}{\nu(G)}$ achieves the upper bound $1$. More precisely, the following is still open:

\begin{problem}
Give a structural characterization of the class of all finite connected graphs $G$ with $\frac{\mu(G)}{\nu(G)} = 1$.
\end{problem}

\section{Definitions and notation}\label{sec:defs}

Here, we define the main notions we use, referring to \cite{West:graph_theory} for basic graph-theoretic terminology.

By a \textbf{graph}, we mean a pair $ G \defeq (V, E)$, where $ V $ is a set of vertices, and $ E $ a set of edges, i.e. $ E \subseteq V^2 $. We will only be considering finite, simple, undirected graphs with no loops, i.e., $(v, v) \notin E$ and $ (u, v) = (v, u) $ for all $u,v \in V$. We identify the corresponding undirected edge by $uv$. A \textbf{matching} in $G$ is a set of edges such that no two are adjacent. A \textbf{perfect matching} is a matching $M$ that covers every vertex, and a \textbf{near perfect matching} is a matching $M$ such that for every vertex $v$ in $G$ except for one, there exists an $e\in M$ such that $v$ is incident to $e$.

Below we define the parameters that we will investigate throughout the
rest of the paper, as well as some additional definitions:



\noindent$\boldsymbol{\lambda(G)} \defeq \max\{|H|+|H'|:(H,H') \text{ are disjoint matchings on } G\}$.
\\
$\boldsymbol{\Lambda(G)}$ is the set of pairs $ (H, H') $ of disjoint
matchings satisfying $|H| + |H'| = \lambda(G)$.
\\
$\boldsymbol{\mu(G)} \defeq \max\{|H| : \exists H' \text{ such that }
(H, H') \in \Lambda(G)\}$.
\\
$\boldsymbol{\Lambda_\mu(G)}$ is the subset of $\Lambda(G)$ of pairs $(H,H')$ with $|H| = \mu(G)$.
\\
$\boldsymbol{\mu'(G)} \defeq \lambda(G) - \mu(G)$.
\\
If $A, B\subseteq E$, then an \textbf{$A, B$ alternating path} is a path such that alternating edges are in $A\setminus B$ and the others in $B\setminus A$. We will frequently refer to $M, H$ alternating paths for the $M$ and $H$ described above.
\\
A path/cycle is \textbf{even} (\textbf{odd}) if it has even (odd) number of edges.
\\
A \textbf{central vertex} of a tree is a vertex with degree larger than 2.
\\
A \textbf{leg} is a copy of $ P_2 $ (the path graph of length $2$, i.e. with two edges). We say $\hat{G}$ is obtained from $G$ by \textbf{adjoining a leg} to a vertex $v$ of $G$ if $\hat{G}$ is obtained by taking the union of $G$ and $P_2$ and identifying one end vertex of $P_2$ with the vertex $v$. 

\noindent Recall that \textbf{spanner} is the tree depicted in \cref{fig:spanner}.
A \textbf{$ \boldsymbol{k} $-spanner} is a tree obtained from the spanner
with $ k $ additional legs adjoined to either of the central vertices, in any combination. Note that the spanner is a 0-spanner. \cref{fig:3spanner} depicts a $3$-spanner.

\begin{figure}[h]
    \centering
    \includegraphics[width=0.5\textwidth]{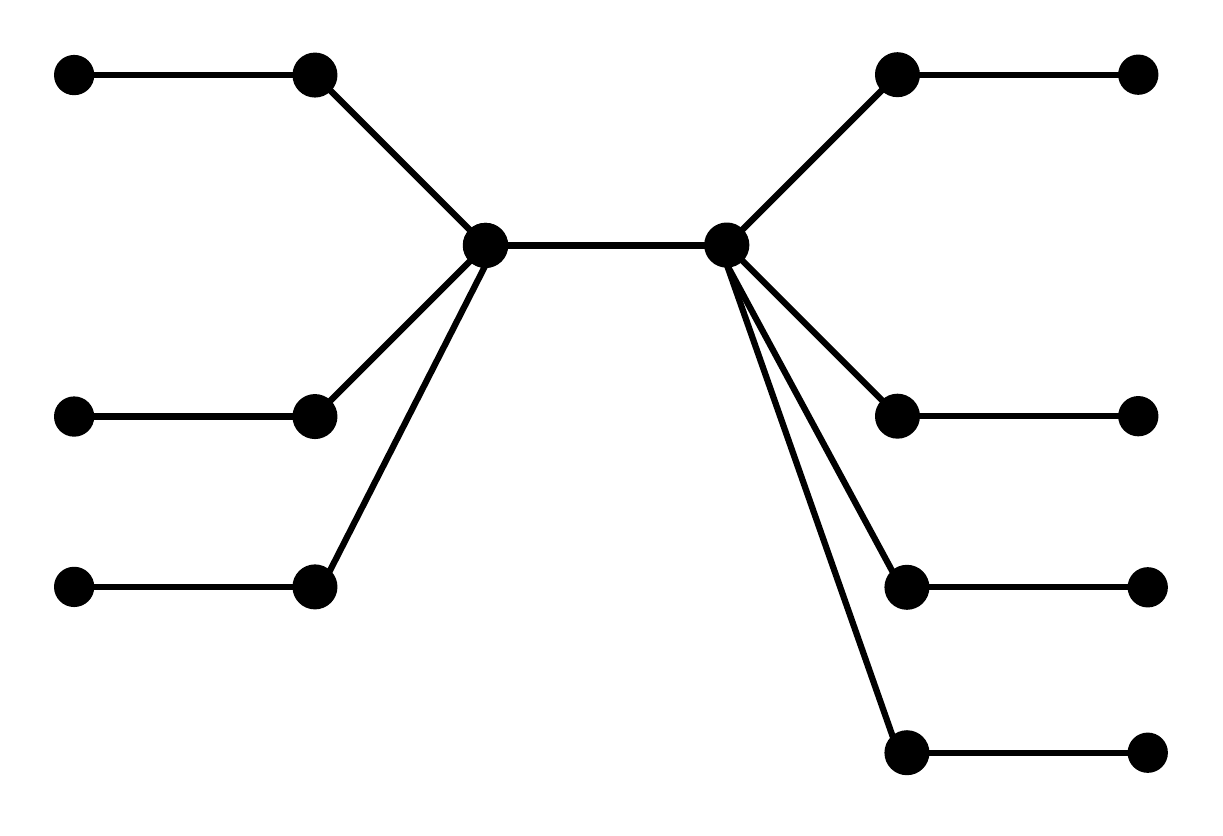}
    \caption{A $3$-spanner.}
    \label{fig:3spanner}
\end{figure}

\noindent An \textbf{inner edge} (resp. \textbf{outer edge}) is an edge in a leg of a
$ k $-spanner adjacent to a central vertex (resp. vertex of degree 1).\\
For a graph $ G \defeq (V, E) $ and a matching $ M \subseteq E $, we say that
a vertex $ v \in V $ is \textbf{saturated} in $ M $ if an edge in $ M $ is
incident to $ v $.


\section{Achieving every ratio}\label{sec:ratios}

In this section, we prove by construction that every rational ratio between $\frac{4}{5}$ and $1$ is achieved by a connected graph (\cref{achieving_every_rational}).

\subsection{Achieving any allowable ratio}
\begin{lemma}
\label{new_edge}
    Let $G \defeq (V, E)$ be a graph with nonadjacent vertices $u$ and $v$.
    Let $\hat{E} \defeq E \cup \{uv\}$, and $\hat{G} \defeq (V, \hat{E})$. Then
    $\lambda(\hat{G}) > \lambda(G) $ if and only if there is some $ (H, H') \in
    \Lambda(G)$ such that $u,v$ are both unsaturated in $H$ or $u,v$ are both
    unsaturated in $H'$.
\end{lemma}
\begin{proof}
    Clearly, $\lambda(\hat{G}) > \lambda(G)$ if and only if $(u,v) \in \hat{H}
    \cup \hat{H}'$ for any $(\hat{H},  \hat{H}') \in \Lambda(\hat{G})$.
    The restriction of $(\hat{H},  \hat{H}')$ to $E$ gives a pair
    in $\Lambda(G)$ with $u,v$ unsaturated in $H$ or $H'$.
    Conversely, if there is $(H,H') \in \Lambda(G)$ such that,
    say, $H$ doesn't saturate $u$ and $v$, then $(H \cup \{(u,v)\}, H') \in
    \Lambda(\hat{G})$.
\end{proof}

\begin{notation*}
For a $k$-spanner $G$, $k > 0$, and a leg $L$ in $G$, we write $G-L$
to denote $G$ with the two vertices incident to the outer edge of $L$ removed. Notice then that $G-L$ is a ($k-1$)-spanner.
\end{notation*}

\begin{lemma}
\label{Leg}
    If $G$ is a $k$-spanner, then for any $(H,H') \in \Lambda_\mu(G)$, the
    inner edges of all but two of the legs adjoined to each central vertex do
    not belong to $H \cup H'$ and the outer edges of those legs are in $H$.
    Furthermore, $(H,H') \in \Lambda_\mu(G)$ is unique, up to
    an automorphism of $G$.
\end{lemma}
\begin{proof}
    At most two of the edges incident to a central vertex can be in
    $H \cup H'$. Thus, $H \cup H'$ must leave out the inner edge of all
    except perhaps two legs adjoined to each central vertex.
    The maximality of $|H \cup H'|$ requires that we include exactly two edges
    incident to each central vertex in $H \cup H'$.
    The outer
    edges of the legs whose inner edges were left out
    must be in $H$ because otherwise, adding each such
    an edge to $H$ and removing it from $H'$ if it was in $H'$ results either
    in an increase of $H$ while preserving $|H \cup H'|$ or in an increase of
    $|H \cup H'|$.
\end{proof}

The proof of \cref{Leg} also establishes the following corollary.

\begin{cor}
\label{Saturated}
    In any $ k $-spanner $ G $ and for any $ (H, H') \in \Lambda_\mu(G) $, we
    have that both of the central vertices are saturated in both $ H $ and
    $ H' $.
\end{cor}

%
\begin{figure}[H]
    \centering
    \includegraphics[width=0.5\textwidth]{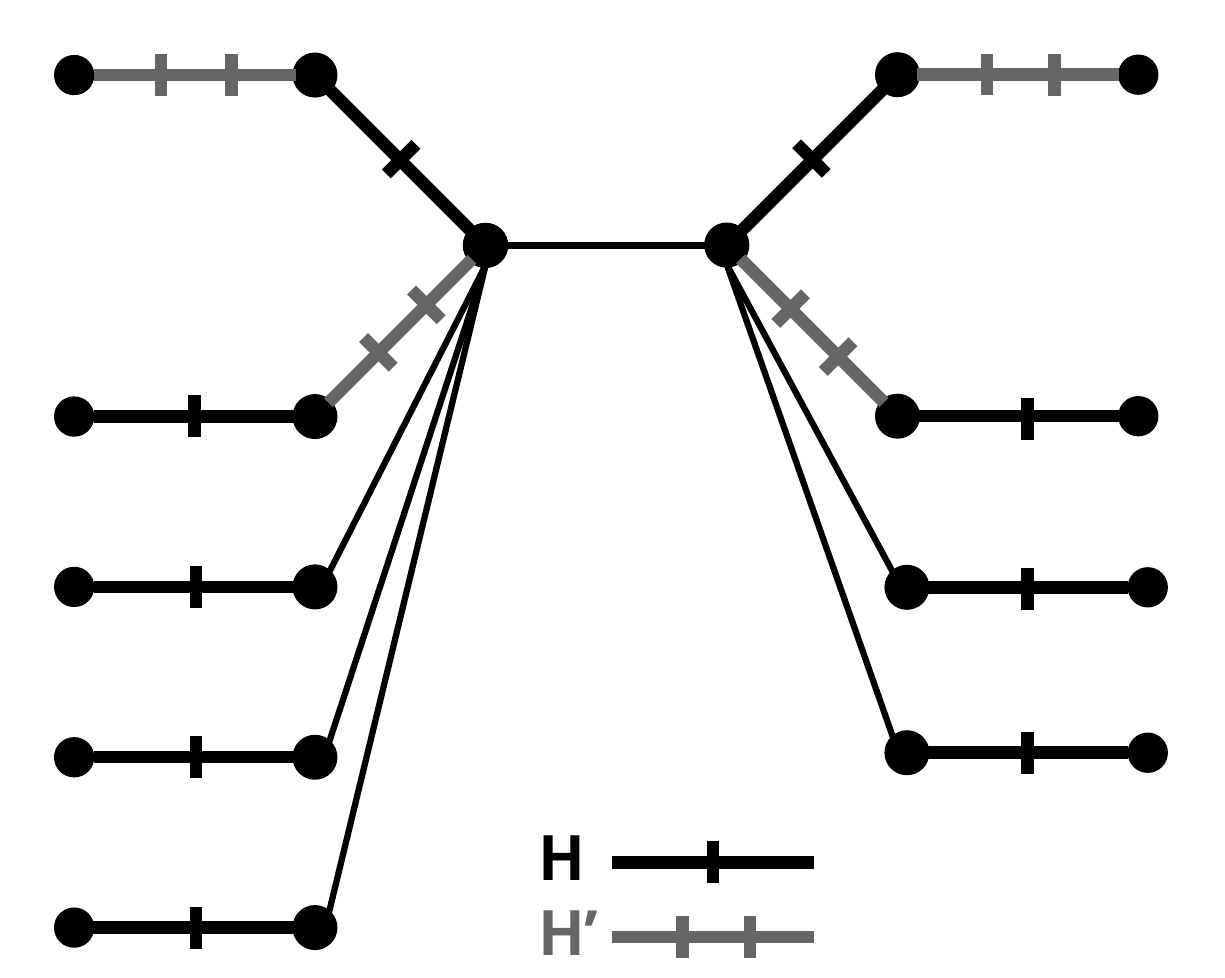}
    \caption{A pair $(H,H') \in \Lambda_\mu$ of a $5$-spanner.}
    \label{fig:5match}
\end{figure}
%

\begin{lemma}
    \label{Matching}
    Let $G$ be a $k$-spanner. Then $G$ admits a unique perfect matching and
    $\nu(G) = k + 5$, $\lambda(G) = k + 8$, and $\mu(G) = k + 4$.
%
%
\end{lemma}
\begin{proof}

    We take all the edges adjacent to the leaves and the unique edge between
    the central vertices. This is clearly a perfect matching and it contains
    one edge per leg, as well as another edge connecting the central vertices,
    so a total of $4 + k + 1$ edges. \cref{fig:span_perf} illustrates this
    perfect matching in a $2$-spanner.
    
    To verify the uniqueness of this perfect matching, note that the edges
    incident to the leaves have
    to be in any perfect matching, therefore the edges adjacent to them cannot
    be. This leaves only the edge between the two central vertices, so it has
    to belong to every perfect matching in order to saturate the central
    vertices.
    
    Fix $(H,H') \in \Lambda_\mu(G)$. By \cref{Leg}, the inner edges of all but
    two legs incident to each central vertex do not belong to $H \cup H'$ and
    the outer edges of those legs are in $H$. Removing these legs from the
    graph results in a $0$-spanner, for which it's easy to see that the
    maximum number of edges from $H \cup H'$ is $8$, where $H$ has $4$ edges,
    and the only possibility
    is to leave the edge between the central vertices out. Each additional leg
    adds one edge to $H$ and none to $H'$, so the formulas follow. For example,
    \cref{fig:5match} illustrates a pair $(H, H')$ in a $5$-spanner.
\end{proof}

\begin{lemma}
\label{Components}
    Let $G_i \defeq (V_i,E_i)$, $i = 1,2$, be distinct graphs, and in each
    graph, suppose there is a vertex $v_i \in V_i$
    that is saturated with both $H_i$ and $H_i'$ for any $(H_i, H_i') \in
    \Lambda(G_i)$. Let $G \defeq (V_1 \cup V_2, E_1 \cup E_2 \cup \{v_1 v_2\})$.
    Then for every $(H,H') \in \Lambda(G)$, $v_1v_2 \notin H \cup H'$. In
    particular, $(H \cap E_i, H' \cap E_i) \in \Lambda(G_i)$ for $i = 1,2$.
    The same statement holds with every occurrence of $\Lambda$ replaced with
    $\Lambda_\mu$.
\end{lemma}
\begin{proof}
    Fix $(H,H') \in \Lambda(G')$ and note that it is enough to show that
    $v_1 v_2 \notin H \cup H'$. Assume towards a contradiction that
    $v_1 v_2 \in H \cup H'$. Then
    $(H \cap E_i, H' \cap E_i) \notin \Lambda(G_i)$ for each $i = 1, 2$
    because $v_i$ is unsaturated with either $H \cap E_i$ or $H' \cap E_i$.
    Thus, $\lambda(G) = |H \cup H'| \le 1 + (\lambda(G_1) - 1) +
    (\lambda(G_2) - 1) = \lambda(G_1) + \lambda(G_2) - 1$. But this
    contradicts the fact that for any $(H_i, H_i') \in \Lambda(G_i)$,
    $(H_1 \cup H_2, H_1' \cup H_2')$ is a pair of disjoint matchings in $G'$
    with total number of edges equal $\lambda(G_1) + \lambda(G_2)$.
\end{proof}

\begin{lemma}
\label{Forest}
    Let $ G_1 \defeq S_1 $ be a $k_1$-spanner graph, and for any $ n \geq 1 $,
    let $ G_{n + 1} $ be a graph formed by connecting one of the central
    vertices in a spanner in $ G_n $ to one of the central vertices of 
    a new $k_n$-spanner $ S_{n + 1} $.
    Then $ \forall m \geq 1 $, we have that if $ (H, H') \in \Lambda_\mu(G_m) $,
    then $(H \cap E(S_n), H' \cap E(S_n)) \in \Lambda_\mu(S_n)$ for each
    $n \le m$.
\end{lemma}
\begin{proof}
    By \cref{Saturated}, the central vertices of every $S_k$ are
    saturated with both $H_k$ and $H_k'$ for any $(H_k,H_k') \in \Lambda(S_k)$.
    The statement, therefore, follows from \cref{Components}.
\end{proof}

We are now ready to prove \cref{achieving_every_rational}, which we restate here for the reader's convenience.

\begin{namedthm*}{\cref{achieving_every_rational}}
    For any $ m, n $ such that $ \frac{4}{5} \leq \frac{m}{n} < 1 $, there is a connected graph $ G $ with $ \mu(G) = m $ and $ \nu(G) = n $.
\end{namedthm*}
\begin{proof}
    Note first that $n \ge m+1$ and $5m - 4n \ge 0$. Let $G \defeq G_{n-m}$ be a graph
    constructed as in \cref{Forest}, where we take a $0$-spanner as $S_1, \dots,
    S_{n-m-1}$ and $k$-spanner as $S_{n-m}$. Then, by Lemmas \ref{Matching}
    and \ref{Forest} we get that
    $\mu(G) = 4(n-m-1) + 4 + k = 4n - 4m + k$ and
    $\nu(G) = 5(n-m-1) + 5 + k = 5n - 5m + k$, so taking
    $k \defeq 5m-4n$ makes $\mu(G) = m$ and $\nu(G) = n$.
\end{proof}

All of the graphs with ratio less than 1 that have been constructed so far
contain a perfect matching by design. However, this is not a necessary
condition for a graph to have ratio less than 1: the graph in
\cref{fig:oddspanner} is a counterexample. It has ratio $\frac{5}{6}$, but it
does not admit a perfect matching.

\begin{figure}[H]
            \centering
		     \includegraphics[scale=0.8]{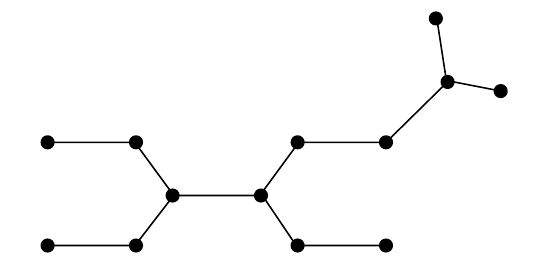}
		     \caption{A graph with an odd number of vertices and $\frac{\mu(G)}{\nu(G)} = \frac{5}{6}$.}
             \label{fig:oddspanner}
\end{figure}
 Furthermore, by adding more vertices around the new 3-vertex, arbitrarily
 many vertices can be missed from the maximum matching. This graph also
 demonstrates that ratio less than 1 can be achieved by a graph with an
 odd number of vertices.\\\\

\section{Ratio 1}\label{sec:ratio_one}

In this section, we provide several sufficient conditions for a graph to have ratio $1$, most notable of which is \cref{admitting_diamond_spanners}. The latter involves \textbf{diamond spanners}, which, by definition, are spanners with any number (possibly zero) of central diamonds as in \cref{fig:forbidden}. We will prove \cref{admitting_diamond_spanners} after some observations and lemmas.

Throughout, we let $G \defeq (V,E)$ be a finite connected graph and $n \defeq |V|$.

\subsection{Observations and examples for ratio 1}

\begin{obs}\label{near_perfect=>maximum}
	$\nu(G) \le \floor{\frac{n}{2}}$.
\end{obs}

\begin{lemma}\label{n-1_covering}
	If $G$ admits two disjoint matchings whose union contains at least $n - 1$ edges, then $\mu(G) = \floor{\frac{n}{2}}=\nu(G)$. 
\end{lemma}
\begin{proof}
	The hypothesis implies that $\lambda(G) \ge n - 1$, i.e., $\lambda(G)$ is $n-1$ or $n$. Either way, $\mu(G) \ge \ceil{\frac{n-1}{2}} = \floor{\frac{n}{2}}$, so $\mu(G) = \nu(G)$ by \cref{near_perfect=>maximum}.
\end{proof}

\begin{defn}
	A path/cycle in $G$ is called \textbf{Hamiltonian} if each vertex of $G$ appears in it exactly once.
\end{defn}

\begin{cor}
	If $G$ admits a Hamiltonian path, then $\mu(G) = \nu(G)$.
\end{cor}
\begin{proof}
	A Hamiltonian path has $n - 1$ edges, which lie in a pair of disjoint matchings, so \cref{n-1_covering} applies.
\end{proof}

		


There are also plenty of graphs with ratio 1 and no Hamiltonian path. 

\begin{example}
In \cref{fig:1nohamilt}, we exhibit an augmented spanner with no Hamiltonian path and having ratio 1.

\begin{figure}[H]
            \centering
		     \includegraphics[scale=0.8]{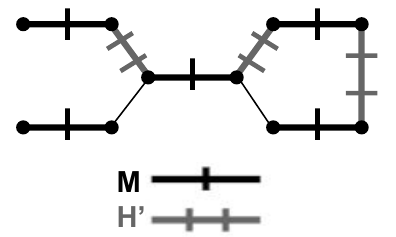}
		     \caption{No Hamiltonian path and $\frac{\mu(G)}{\nu(G)} = 1$.}
             \label{fig:1nohamilt}
\end{figure}

\noindent
Indeed, the graph $G$ in \cref{fig:1nohamilt} has $\lambda(G) = 8$, and this can be achieved with $|H| = 5$ and $|H'| = 3$, as in the figure. $H$ is also a perfect matching, so $\nu(G) = \mu(G) = 5$.
\end{example}

\begin{example}
Here is another example with ratio 1 and no Hamiltonian path. Let $P_{n,2}$ be the \textbf{propeller graph} with $n$ blades of length $2$, namely, the graph obtained by joining $n$ paths of length two to a shared central vertex.
\Cref{fig:propeller} depicts $P_{4,2}$.
\begin{figure}[H]
            \centering
		     \includegraphics[scale=0.5]{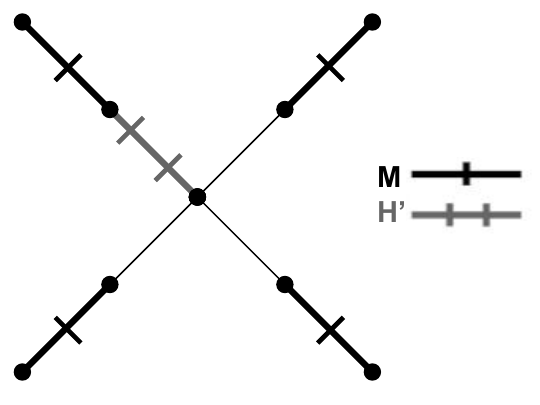}
		     \caption{$P_{4,2}$.}
             \label{fig:propeller}
\end{figure}
\begin{prop}
$\lambda(P_{n,2}) = n+2$, $\nu(P_{n,2}) = \mu(P_{n,2}) = n$, and hence $\mu'(P_{n,2}) = 2$. In particular, the difference between the parameters $\mu$ and $\mu'$ can be arbitrarily large.

\end{prop}
\begin{proof}
For any pair $(H,H')$ of disjoint matchings, each leaf can be incident to at most one edge in $H \cup H'$ and the central vertex can be incident to at most two edges in $H \cup H'$, so $\lambda(P_{n,2}) \le n+2$. We show that this is, in fact, an equality. Letting $u$ denote the central vertex, $w$ be a leaf, and $v$ be the vertex adjacent to both $u$ and $w$, we let $H$ contain $u v$ and all edges incident to all leaves except for $w$. We then let $H'$ contain $v w$ and an edge incident to $u$ distinct from $u v$. Thus $|H| = n$ and $|H'| = 2$, so $(H,H')$ achieves $\lambda(P_{n,2})$. By Observation 5.2, $H$ is a maximum matching, so $\nu(P_{n,2}) = \mu(P_{n,2}) = n$.
\end{proof}
\end{example}

\subsection{Alternating paths}
We present some lemmas about alternating paths before we prove \cref{admitting_diamond_spanners}.

\begin{defn}
Choose matchings $(M,H,H')$ such that $|M|=\nu(G)$, $|H|=\mu(G)$, $H$ is disjoint from $H'$, $|H|+|H'|=\lambda(G)$, $|M\cap(H\cup H')|$ is maximized, and among these, such that $|M\cap H|$ is maximized. Call the triple $(M,H,H')$ a \textbf{maximum intersection triple}.
\end{defn}
\begin{lemma}\label{augpaths}
Let $G$ be a finite graph with $\mu(G)<\nu(G)$. Let $(M,H,H')$ be a maximum intersection triple of $G$. Let $P$ be a maximal $M-H$ alternating chain, and it follows from maximality of $(M,H,H')$ that $P$ is length at least two. Then
\begin{enumerate}[label=(\roman*)]

\item $P$ is not a cycle.
\item The length of $P$ is odd and its end edges are in $M$.
\item The end edges of $P$ are in $H'$.
\item All interior vertices of $P$ are incident to $H'$.
\item $P$ contains at least one $M\backslash H'$ edge.
\item There does not exist an even cycle $C$ in $G$ such that half of $C$'s edges are in $M$, the remaining edges in $H\cup H'$, and $C\cap P \cap M\neq \emptyset$.
\end{enumerate}
\end{lemma}
\begin{proof}

(i): Suppose $P$ is a cycle. If $P$ is an odd cycle then note that two adjacent edges will be in the same matching, contradicting matching definition. Suppose then that $P$ is an even cycle. Consider $M'\defeq(M\backslash P)\cup(H\cap P)$, and note that this is a matching since $P$ is an even cycle and maximum since $M$ is. Then $|M'\cap(H\cup H')| \geq |M\cap(H\cup H')|$ and $|M'\cap H|>|M\cap H|$, a contradiction.\\\\
(ii): By (i), $P$ is not a cycle, so $P$ is a path. Suppose that some end edge of $P$, $e_0$ is in $H$. $e_0$ can be incident to $M$ on only one side, as $P$ is a maximal alternating path. Call this $M\setminus H$ edge $e_1$. Since $e_0$ is incident to $M$ on only one side, we may define $M'=(M\backslash\{e_1\})\cup \{e_0\}$. $M'$ is a maximal matching, but $|M'\cap H| > |M\cap H|$, contradicting that ($M,H,H'$) was a maximal intersection triple. $P$ thus has no end edges in $H$. This implies that $P$ also has odd length.\\\\
(iii): Consider an end edge $e_0$ in $M\backslash H$. Again note that $e_0$ can be incident to $H$ on only one side. If $e_0$ is not in $H'$, then we can remove $e_1$, the preceding edge, from $H$, add $e_0$ to $H$, and again increase $|M\cap(H\cup H')|$.\\\\
(iv): Suppose instead some interior vertex is not incident to $H'$, and note that this vertex is incident to no end edges. This vertex is incident to exactly one edge in $M$, which we call $e_M$. $e_M$ is incident to $H'$ at most once. Via a similar argument to above, if we take $e_M$ to be in $H'$ instead of this adjacent edge, we increase $|M\cap(H\cup H')|$.\\\\
(v): Suppose towards a contradiction that all $M$ edges in $P$ are also in $H'$. Swapping $H$ and $H'$ on the path will increase $|M\cap H|$ and keep $|M\cap(H\cup H')|$ the same, a contradiction. We are allowed to do this because the two end edges of $P$ are not connected to $H$.\\\\
(vi): Suppose there exists such a $C$. Observe that $C$ must contain some $M\backslash (H\cup H')$ edge. To see this, note that because $P$ is acyclic, some $C\cap P \cap M$ edge, say $e$, is incident to fewer than two $C\cap P \cap H$ edges. Then $e$ must be adjacent to a $C\cap H'$ edge, or else $P$ is not maximal. Then $e\notin H'$, and $e\notin H$ since $e\in P$. Now, define $M' \defeq (M\backslash C)\cup (C\backslash M)$. $M'$ is a matching since no vertex of $C$ is adjacent to an $M$ edge not already in $C$. $M'$ is maximal since $M$ is. Because $C$ contained some $M\backslash (H\cup H')$ edge, immediately $|M'\cap(H\cup H')|>|M\cap(H\cup H')|$.
\end{proof}

\subsection{Sufficient condition for ratio 1: forbidden subgraphs}\label{sec:forb}
Let the \textbf{diamond graph} be the graph with four vertices depicted
in \cref{fig:diamond}.\\
Let a \textbf{diamond chain} be a graph resulting from taking some
natural number $n$ copies of the diamond graph, indexing them $D_i$.
Arbitrarily
distinguishing between the two vertices of degree two in each diamond
as the ``left'' and ``right'' such vertices, add an edge between
the right degree two vertex of diamond $D_i$ and the left degree two
vertex of $D_{i+1}$ for $1 \le i \le n-1$. A diamond chain is depicted in \cref{fig:diamond_chain}

%
\begin{figure}[h]
    \begin{subfigure}[b]{0.2\textwidth}
    \centering
    \includegraphics[width=0.75\textwidth]{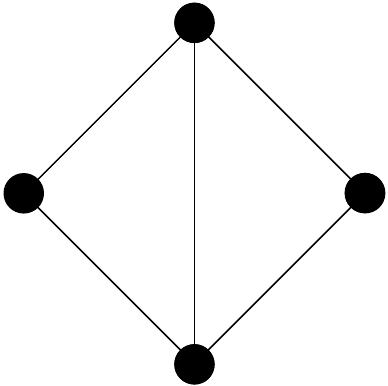}
    \caption{The diamond.}
    \label{fig:diamond}
    \end{subfigure}
    \\
    \begin{subfigure}[b]{0.8\textwidth}
    \includegraphics[width=\textwidth]{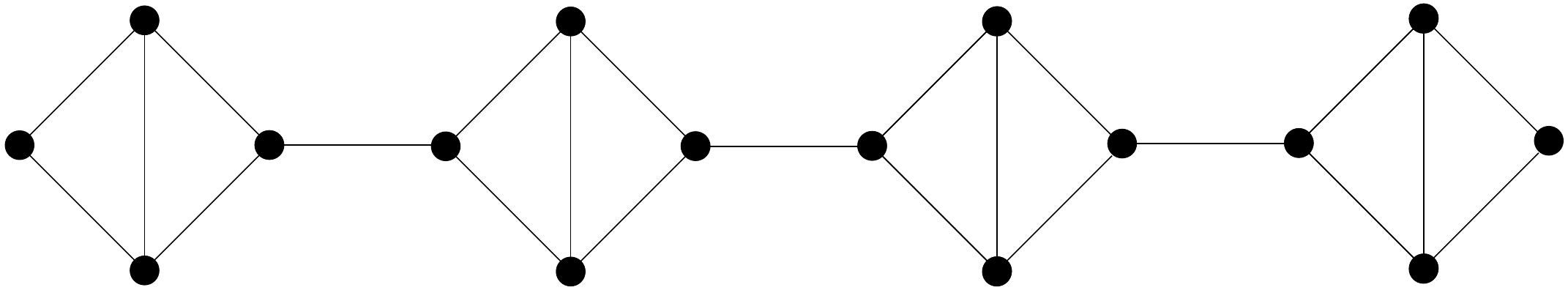}
    \caption{A diamond chain with $4$ diamonds.}
    \label{fig:diamond_chain}
    \end{subfigure}
    \caption{}
\end{figure}

\begin{defn}
\label{def:diamond_spanner}
A \textbf{diamond spanner} is a graph obtained from the spanner via
removing the edge between the central vertices and taking disjoint union with
a diamond chain. The central vertices of the spanner are then
connected to the degree two end vertices of the diamond chain. We also consider the
spanner to be a diamond spanner. See \cref{fig:forbidden}.
\end{defn}
We prove one more lemma about $M$-$H$ alternating paths before proving
\cref{admitting_diamond_spanners}.
\begin{lemma}\label{barriers}
Assume the same conditions as in \cref{augpaths}, and let $P$ be a maximal $M-H$ alternating path. If $P$ contains an edge $e$ of $M\setminus H'$ such that $G$ has no $H'$ edge or $H'-M$ alternating path connecting the two components of $P\setminus e$, then $e$ is an edge of a diamond spanner of $G$.
\end{lemma}
\begin{figure}[H]
            \centering
		     \includegraphics[scale=0.7]{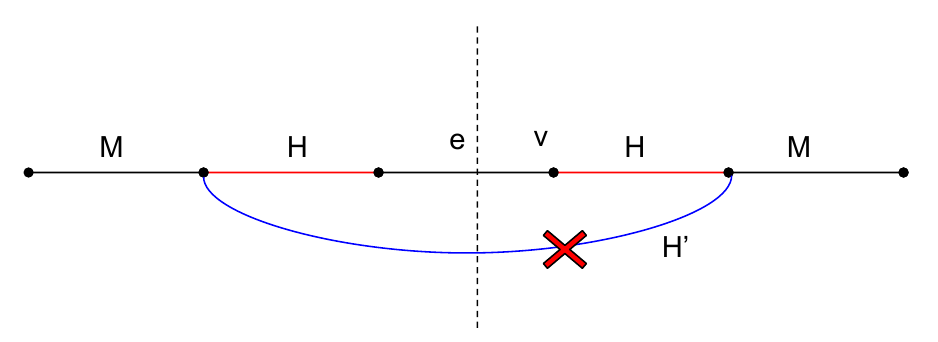}
		     \caption{The setup described in \cref{barriers}. No $H'$ edge or $H'-M$ alternating paths may cross the dashed barrier through the given $M\setminus H'$ edge.}
             \label{fig:barrier}
\end{figure}
\begin{proof}
The setup is depicted in \cref{fig:barrier}. We will show that each piece of $P\setminus e$ must witness a chain of diamonds that terminates in a spanner-end. Observe first that $e$ is not an end edge of $P$ by (iii) of \cref{augpaths}. Consider the right vertex $v$ of $e$. By (iv) of \cref{augpaths}, it must be adjacent to an $H'$ edge. By assumption, this edge does not cross over to the left side of $e$. Either this $H'$ edge connects $v$ to some vertex not in $P$ or to another $P$ vertex to the right of $e$. In the first case, note that this $H'$ edge must be connected to an $M$ edge, or else $|M\cap(H\cup H')|$ is not maximal. This $M$ edge must also be off of $P$ since all $P$ edges are saturated with $M$. We now have a path of length at least 2 extending off of $P$, and we see that the right side of $e$ is now a ``spanner-end", as in \cref{fig:spannerend1}.
\begin{figure}[H]
            \centering
		     \includegraphics[scale=0.7]{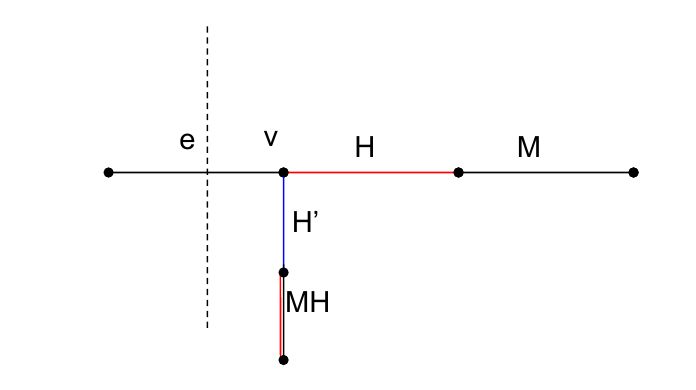}
		     \caption{Half of a spanner on the right side of $e$.}
             \label{fig:spannerend1}
\end{figure}
If instead our $H'$ edge connects back to the right of $v$ on $P$, then if it connects to a vertex of distance more than 2 away then we again have a spanner end, as shown in \cref{fig:spannerend2}. 
\begin{figure}[H]
    \centering
     \includegraphics[scale=0.7]{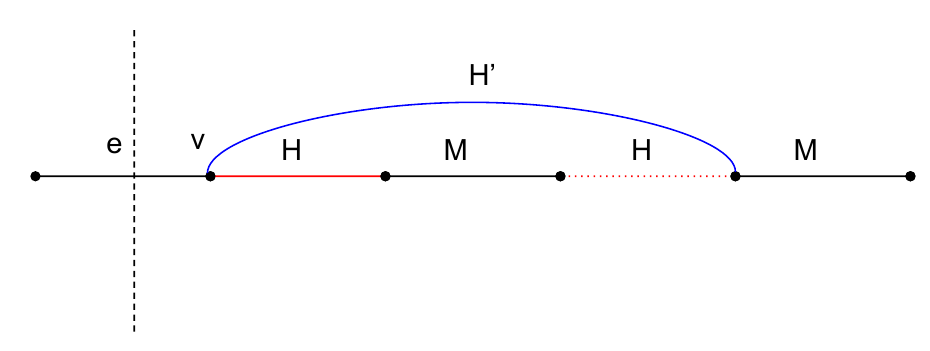}
     \caption{The dashed $H$ edge is not part of the spanner end.}
     \label{fig:spannerend2}
\end{figure}
If instead it connects to the vertex 2 to the right of $v$, then if follows from (iv) of \cref{augpaths} that the vertex directly to the right of $v$ must also be connected to an $H'$ edge. Again, if this edge connects more than distance 2 to the right, or goes off the path, we have a spanner-end. If instead this second $H'$ edge connects distance 2 away, observe that we now have a diamond configuration on the right hand side of $e$. This is illustrated in \cref{fig:diamondpath}.
\begin{figure}[H]
            \centering
		     \includegraphics[scale=0.7]{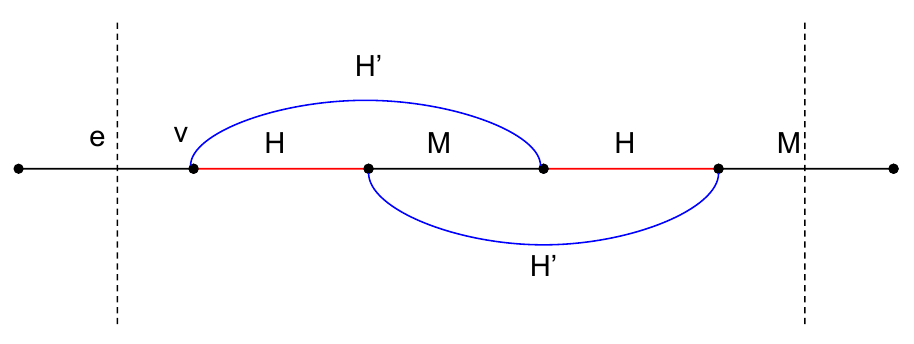}
             
                \caption{A diamond configuration if $H'$ is linked distance 2 away.}
                \label{fig:diamondpath}
\end{figure}
Notice that the $M$ edge to the right of this diamond cannot be in $H'$, and that any $H'$ or $H'$-$M$-$H'$ path crossing it would have to cross $e$ as well. Thus, we may repeat our argument for this new $M\backslash H'$ edge. Since $G$ is finite, we must eventually choose a spanner end on the right side. We extend this argument to the left side of $e$ by noting that any $H'$ edge going off of $P$ on the left cannot be incident to any $H'$-$M$ path off of $P$ on the right of $e$ by assumption.
\end{proof}

We may now prove \cref{admitting_diamond_spanners}, which we restate here for the reader's convenience.

\smallskip

\begin{namedthm*}{\cref{admitting_diamond_spanners}}
If a finite connected graph $G$ has $\frac{\mu(G)}{\nu(G)}<1$, then it
contains a diamond spanner as a subgraph.
\end{namedthm*}
\begin{proof}
Assume the same conditions and fix $P$ as in $\cref{barriers}$. Note that such an $M-H$ alternating chain must exist, or else $\mu(G)=\nu(G)$. Then $P$ contains some $M\backslash H'$ edge by (v) of \cref{augpaths}. Consider the leftmost of these edges $e_0$. The left vertex $v_0$ of $e_0$ must be incident to some $H'$ edge. If this edge does not connect $v_0$ to some vertex on the right of $v_0$ and is not part of an $H'$-$M$-$H'$ path connecting to the right of $v_0$, then \cref{barriers} implies the existence of a spanner or diamond spanner. Otherwise, we may assume it is an $H'$ edge connecting back to $P$, since we will derive contradictions by the illegal even cycles described in (vi) of \cref{augpaths}. This $H'$ edge connects to some vertex to the right of $v_0$, observe that if this $H'$ connects directly to the left of some $M$ edge that we have created an even cycle forbidden by \cref{augpaths}. This cycle is pictured in \cref{fig:illegal}.
\begin{figure}[H]
            \centering
		     \includegraphics[scale=0.7]{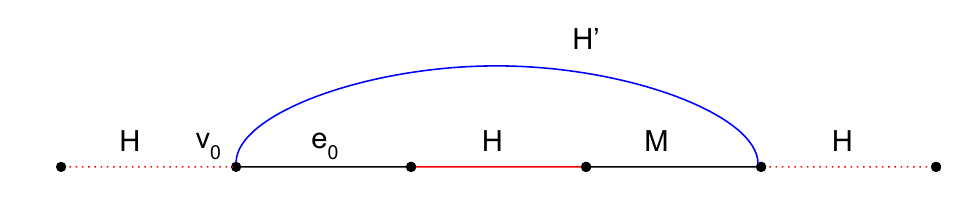}
		     \caption{An illegal even cycle.}
             \label{fig:illegal}
\end{figure}
This $H'$ edge must then connect to the left of some $M$ edge $e_1$ by a vertex $v_1$. Note $e_1$ cannot be in $H'$. If none of the vertices between $v_0$ and $v_1$ are linked by an $H'$ or $H'$-$M$-$H'$ path to some vertex on the right of $v_1$, then \cref{barriers} again gives us a spanner or diamond spanner. The barrier is drawn in \cref{fig:anotherbarrier}. 
\begin{figure}[H]
            \centering
		     \includegraphics[scale=0.7]{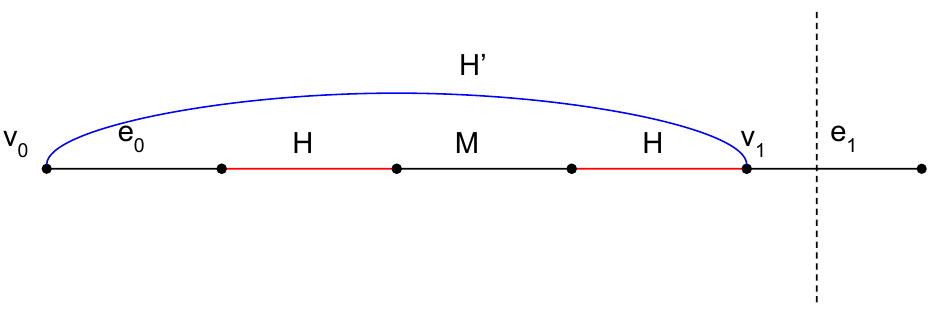}
             
                \caption{A barrier if no other vertices are connected to the right of $v_1$.}
                \label{fig:anotherbarrier}
\end{figure}
Suppose instead that one of these vertices is linked to the right of $v_1$. Choose the vertex between $v_0$ and $v_1$ that is connected by such a path that travels the furthest to the right. Call this vertex $v'_2$ and the vertex it is linked to $v_2$. If $v'_2$ is on the left of an $M$ edge and $v_2$ on the right of an $M$ edge, then the two paths between $v_2$ and $v_2'$ form another prohibited even cycle. If $v'_2$ is on the right of an $M$ edge and $v_2$ on the right of an $M$ edge, then the cycle connecting $v'_2$ to $v_2$ to $v_1$ to $v_0$ is also prohibited. This forbidden cycle is shown in \cref{fig:anotherbadcycle}.
\begin{figure}[H]
            \centering
		     \includegraphics[scale=0.5]{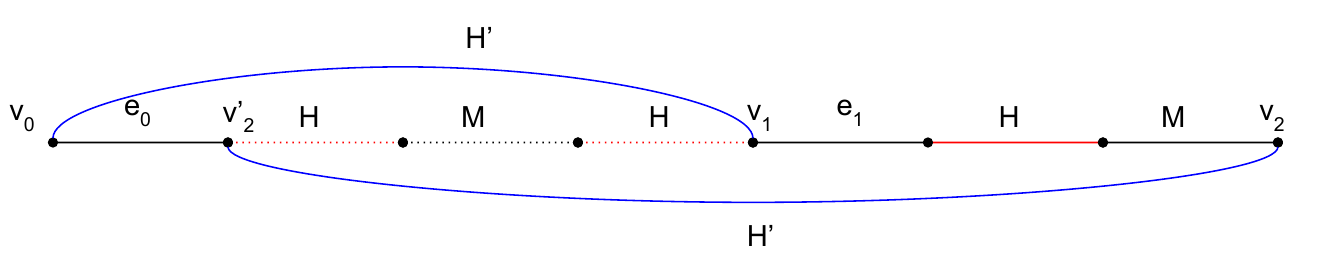}
		     \caption{An illegal cycle using all 4 marked vertices.}
             \label{fig:anotherbadcycle}
\end{figure}

Thus, $v_2$ must also be on the left of some $M\backslash H'$ edge $e_2$. By a similar argument to the case for $v_1$, if no vertex between $v_0$ and $v_2$ is linked to a vertex of $P$ on the right of $v_2$ by an $H'$ or $M$-$H'$-$M$ path, then \cref{barriers} applies. Thus, some vertex between $v_0$ and $v_2$ must be linked to the right of $v_2$. In fact, by choice of $v_2$ being furthest away of all linked vertices from those between $v_0$ and $v_1$, we may restrict out attention to the vertices between $v_1$ and $v_2$. Among these, again choose the vertex $v_3'$ that is linked the furthest to the right to some vertex $v_3$, as shown in \cref{fig:longcross}.
\begin{figure}[H]
            \centering
		     \includegraphics[scale=0.45]{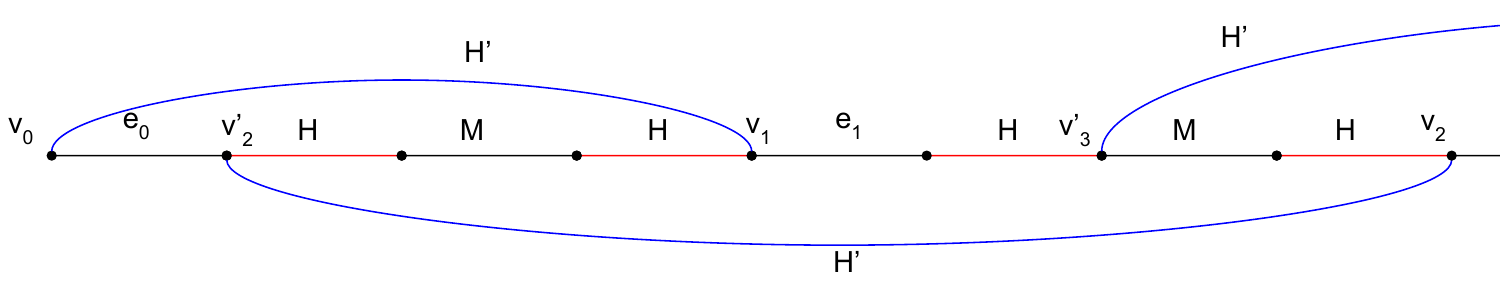}
             
                \caption{We continue our argument by finding another vertex $v_3'$ that is linked to the right on our path.}
                \label{fig:longcross}
\end{figure}

Again, if $v_3'$ is to the left of an $M$ edge and $v_3$ to the right, then these two paths between these two vertices form a forbidden even cycle. If instead $v_3'$ is on the right of an $M$ edge and $v_3$ on the right, then these two vertices are part of another forbidden even cycle, regardless of whether $v_2'$ is on the left or right of an $M$ edge. Thus, $v_3$ is on the left of an $M$ edge $e_3$. We may continue arguing in this manner, but observe that there are only finitely many $M\backslash H'$ edges in $P$, so eventually, we will not be able to cross some $M\backslash H'$ edge $e_n$ without creating a prohibited cycle. No vertex to the right of $e_n$ can be linked to a vertex to the left since our $H'$ and $H'$-$M$-$H'$ paths were chosen to reach the furthest. Then $e_n$ witnesses a spanner or diamond spanner by \cref{barriers}.
\end{proof}

\appendix
\section{Characterization of Graphs Achieving Ratio $\frac{4}{5}$}
\label{app:char}

For the reader's convenience, here we provide a short exposition of Tserunyan's characterization \cite{Tserunyan:2009qr} of exactly when the lower bound $\frac{4}{5}$ is achieved by a connected graph.

Let $S$ denote the spanner as a concrete graph, and note that $S$ is a tree. Call
a vertex of the spanner a $j$-vertex if it has degree $j$. For a $1$-vertex or
$2$-vertex $v$, call the nearest $3$-vertex the \emph{base} of $v$.
Call an edge between an $i$-vertex and a $j$-vertex an $i-j$ edge.
Define two sets associated to the spanner $S$:
\begin{align*}
    U(S) &:= \{e \in E(G) \ :\ \text{$e$ is incident to a $1$-vertex} \}\\
    L(S) &:= \{e \in E(G)\setminus U(S) \ :\ \text{$e$ is incident to a
    $2$-vertex} \}
\end{align*}
For an arbitrary graph $G$,
call a spanning subgraph where every connected component is isomorphic to $S$
an $S$-forest.
Let $G$ admit an $S$-forest $F$. Denote the connected
components $S_1, S_2, \ldots, S_k$, and define
\begin{align*}
    U(F) &:= \bigcup_{j=1}^k U(S_j)\\
    L(F) &:= \bigcup_{j=1}^k L(S_j)\\
    \Delta(G, F) &:= \{e \in E(G) \ :\ \text{$e$ connects a $1$-vertex to its base}\}\\
    B(G,F) &:= E(G) \setminus \Big ( U(F) \cup L(F) \cup \Delta(G,F) \Big ).
\end{align*}
Note that $U, L$ are subsets of $E(F)$ whereas $\Delta, B$ are subsets of $E(G)$.

Call a non-repeating sequence of edges $(e_1, e_2, \ldots e_k)$ such that
$e_i$ is adjacent to $e_{i+1}$ for all $i$ a \emph{trail}.
If $u_{i-1}, u_i$ are the vertices incident to $e_i$, call a trail
\emph{closed} if $u_0 = u_k$.
For sets $A,B \subseteq E(G)$, say a trail $T = (e_1, e_2, \ldots, e_k)$ is
$A,B$-alternating if the edges of even index belong to $A \setminus B$ and those of
odd index belong to $B \setminus A$ or vice-versa.
\begin{theorem}
A connected graph $G$ has ratio $\frac{4}{5}$ if and only if $G$ admits a spanning
$S$-forest $F$ such that
\begin{enumerate}
    \item $1$-vertices of $F$ are not incident to any edge from $B(G,F)$
    \item if $u$ is a $1$-vertex and is incident to an edge of $\Delta(G,F)$, then
    the $2$-vertex adjacent to $u$ is not incident to any edges of $B(G,F)$
    \item for any $L(F), B(G,F)$-alternating closed trail $C$ containing a $2-2$
    edge, we have $C \cap B(G,F)$ is not bipartite.
\end{enumerate}
\end{theorem}

\medskip

\subsubsection*{Acknowledgements}
We thank Vahan Mkrtchyan for his invaluable insight into this direction of research and for pointing out the relevant literature and context. We also thank Alexandr Kostochka for helpful suggestions and advice.

\bibliographystyle{ieeetr}
\bibliography{references}\vspace{0.75in}

\end{document}